\documentclass[11pt]{article}
\usepackage{amsmath}
\usepackage{amsfonts,amsthm,amssymb}
\usepackage{amsfonts}
\usepackage{graphicx}    
\usepackage{subcaption}  
\usepackage{epstopdf}
\usepackage{caption} 
\usepackage{indentfirst}
\usepackage{amssymb}
\usepackage{hyperref}
\newtheorem{lem}{Lemma}[section]
\newtheorem{thm}[lem]{Theorem}
\newtheorem{cor}[lem]{Corollary}
\newtheorem{conj}{Conjecture}

\newtheorem{Def}[lem]{Definition}
\newtheorem{Obs}[lem]{Observation}
\newtheorem{Cla}{Claim}

\allowdisplaybreaks[4]

\begin{document}
	
	\title{Extremal graphs with no subgraph admitting $k+1$ edge-disjoint spanning trees\footnote{The research is supported by Natural Science Foundation of Xinjiang Uygur Autonomous Region (2025D01E02) and National Natural Science Foundation of China (12261086).}}
	\author{Qinglin Wang, Yingzhi Tian\footnote{Corresponding author. E-mail: wqlxju@163.com (Q. Wang), tianyzhxj@163.com (Y. Tian).} \\
		{\small College of Mathematics and System Sciences, Xinjiang
			University, Urumqi, Xinjiang 830046, PR China}}
	\date{}
	\maketitle
	
	\begin{abstract}
	 A graph $G$ is $\tau_k$-maximal if $G$ contains no subgraph admitting $k+1$ edge-disjoint spanning trees, while the addition of any edge in the complement of $G$ yields a subgraph that admits $k+1$ edge-disjoint spanning trees. In this paper, we prove that for any integers $k\geq 1$ and $n\geq 2k+2$, every $\tau_k$-maximal graph of order $n$ satisfies $|E(G)|\leq (k+1)(n-1)-1$. Furthermore, we construct a family of $\tau_k$-maximal graphs on $n\ge 2k+2$ vertices that have exactly $(k+1)(n-1)-1$ edges, which establishes the tightness of the upper bound. Then we conjecture that every $\tau_k$-maximal graph on $n$ vertices has exactly $(k+1)(n-1)-1$ edges, and we verify the conjecture for the case $k=1$.
	\end{abstract}
	
	
    \section{Introduction}\label{Se1}
	
	All graphs considered in this paper are \emph{finite, simple}, and \emph{undirected}. For terminology and notation not explicitly defined here, we refer the reader to Bondy and Murty~\cite{Bondy2008}. 
	
    Let $G$ be a graph with \emph{vertex set} $V(G)$ and \emph{edge set} $E(G)$. The \emph{edge connectivity} of $G$, denoted by $\kappa'(G)$, is the minimum number of edges whose removal disconnects $G$. The \emph{complement} of $G$, denoted by $G^c$, is the graph with vertex set $V(G^c)=V(G)$ and edge set $E(G^c)=\{uv:u,v\in V(G),~uv\notin E(G)\}$. 
	
	Given a subset $F\subseteq E(G^c)$, let $G+F$ denote the graph with vertex set $V(G)$ and edge set $E(G)\cup F$. In particular, we write $G+e$ for $G+\{e\}$. For a subset $X\subseteq V(G)$, $G[X]$ denotes the \emph{induced subgraph} of $G$ on $X$. For two vertex subsets $S$, $T\subseteq V(G)$, let $E_G[S,T]$ denote the set of edges of $G$ with one end-vertex in $S$ and the other in $T$. When $T=V(G)\setminus S$, we abbreviate $E_G[S,V(G)\setminus S]$ to $E_G(S)$. In particular, we write $E_G(v)$ for $E_G(\{v\})$.
	   
	For a graph $G$ and a partition $\mathcal{P}=\{P_1,\dots,P_t\}$ of $V(G)$, let $E_{\mathcal{P}}(G)$ denote the set of edges of $G$ joining vertices in distinct parts of $\mathcal{P}$. Define $G_{\mathcal{P}}$ as the graph with vertex set $\mathcal{P}$ and edge set $E_{\mathcal{P}}(G)$, where each edge connects the two parts of $\mathcal{P}$ that contain its end-vertices in $G$. In other words, $G_{\mathcal{P}}$ is the graph obtained from $G$ by contracting each member of $\mathcal{P}$ to a single vertex.
	                                
	For a graph $G$, a \emph{spanning subgraph} of $G$ is a subgraph $H$ with $V(H)=V(G)$ and $E(H)\subseteq E(G)$. A \emph{spanning tree} of $G$ is a spanning subgraph that is a tree. A family of subgraphs of $G$ is \emph{edge-disjoint} if no two of them share an edge. The \emph{spanning tree packing number} of a graph $G$, denoted by $\tau(G)$, is the maximum number of edge-disjoint spanning trees contained in $G$. 
	
    The study of spanning tree packing number is of fundamental importance in graph theory. It has wide applications in network design for enhancing reliability and transmission efficiency~\cite{Itai1988}, and is closely related to several combinatorial properties of graphs, such as rigidity~\cite{Tay1984}, circular flow~\cite{Lai2007} and collapsibility~\cite{Catlin1988}.  
    
    A cornerstone result in this field is the celebrated characterization of graphs admitting at least $k$ edge-disjoint spanning trees, which was independently introduced by Nash-Williams~\cite{Nash1961} and Tutte~\cite{Tutte1961}, and serves as a key tool in our investigation.
    
    \begin{thm}[Nash-Williams~\cite{Nash1961} and Tutte~\cite{Tutte1961}]\label{NT}
    	A graph $G$ admits $k$ edge-disjoint spanning trees if and only if $|E(G_{\mathcal{P}})| \geq k(|\mathcal{P}|-1)$ for every partition $\mathcal{P}$ of $V(G)$. 
    \end{thm}
    
     The spanning tree packing number is also closely associated with edge-connectivity. Kundu~\cite{Kundu1974} established the inequalities $\lceil\frac{\kappa'(G)-1}{2}\rceil\leq \tau(G)\leq \kappa'(G)$. Building on this relationship, Gusfield~\cite{Gusfield1983} later obtained a simpler sufficient condition in terms of edge connectivity for the existence of $k$ edge-disjoint spanning trees. For a comprehensive survey on the spanning tree packing number, we refer the reader to Palmer~\cite{Palmer2001}.
     
     \begin{thm}[Gusfield~\cite{Gusfield1983}]\label{GT}
     	If $G$ is a graph with $\kappa'(G) \geq 2k$, then $G$ admits at least $k$ edge-disjoint spanning trees.
     \end{thm}
     
    In this paper, we are concerned with the graph parameter $\bar{\tau}(G)=\max\{\tau(H): H\subseteq G\}$. Observe that $\bar{\tau}(G)$ is monotone with respect to edge addition: \[\bar{\tau}(G)\leq\bar{\tau}(G+e)\leq \bar{\tau}(G)+1.\]
    A central theme in extremal graph theory is the study of graphs that are maximal with respect to a given monotone graph property. Formally, a graph is called maximal with respect to a property $\mathcal{Q}$ if it does not satisfy $\mathcal{Q}$, but the addition of any new edge creates a graph that does. Determining the extremal size of such maximal graphs is a classic extremal problem. In this paper, we investigate the sizes of maximal graphs with respect to $\bar{\tau}(G)$.
	
	Let $k$ be a positive integer. A graph $G$ is \emph{$\tau_k$-maximal} if $\bar{\tau}(G)\leq k$, but for any edge $e\in E(G^c)$, we have $\bar{\tau}(G+e)\geq k+1$. Since $\tau(K_n)=\lfloor\frac{n}{2}\rfloor$, it follows that $\bar{\tau}(K_n)\leq k$ for all $n\leq 2k+1$. Therefore, we note that $K_n$ with $n\leq 2k+1$ is $\tau_k$-maximal. The main goal of this paper is to determine the size of $\tau_k$-maximal graphs of order $n\geq 2k+2$.
	
	The remainder of this paper is organized as follows. Section~\ref{Se2} presents several properties of $\tau_k$-maximal graphs. In Section~\ref{Se3}, we give a universal upper bound on the size of $\tau_k$-maximal graphs. Then we construct a family of $\tau_k$-maximal graphs with prescribed edge connectivity that achieves this bound. Motivated by the construction and the upper bound, we propose a conjecture on the size of general $\tau_k$-maximal graphs and verify this conjecture for the case $k=1$.
	
    \section{Properties of $\tau_k$-maximal graphs}\label{Se2}	

    In this section, we present several properties of $\tau_k$-maximal graphs that will be used throughout the remainder of this paper. We start with an upper bound on the size of graphs satisfying $\bar{\tau}(G)<k$.
    
	\begin{lem}\label{size-bound}
		Let $k \geq 2$ and $n \geq 2$ be integers, and let $G$ be a graph of order $n$. If $\bar{\tau}(G) < k$, then $|E(G)| < k(n-1)$.
	\end{lem}
	\begin{proof}
		We proceed by induction on $n$. For $n=2$, we have $|E(G)| \leq 1 <k =k(n-1)$, so the result holds trivially.
		
		Now assume $n \geq 3$ and that the result holds for all graphs of order fewer than $n$. Since $\tau(G) \leq \bar{\tau}(G) < k$, it follows from Theorem~\ref{NT} that there exists a partition $\mathcal{P}=\{P_1, P_2,\ldots, P_t\}$ of $V(G)$ such that \[
        |E(G_{\mathcal{P}})| < k(t-1).\]
		
		For each $i=1,2,\ldots,t$, let $|P_i|=n_i$. If $n_i=1$, then $|E(G[P_i])|=0=k(n_i-1)$. If $n_i\geq 2$, then $\bar{\tau}(G[P_i])\leq \bar{\tau}(G)<k$. By the induction hypothesis, we get \[|E(G[P_i])|<k(n_i-1).\] Consequently,
		\begin{align*}
			|E(G)|&=\sum_{i=1}^t |E(G[P_i])|+|E(G_{\mathcal{P}})| \\
            &< \sum_{i=1}^t k(n_i-1) + k(t-1) \\
			&= k(n-t)+k(t-1)\\
			&=k(n-1),
		\end{align*}  
		which completes the proof.
	\end{proof} 
	
	The converse of Lemma~\ref{size-bound}, namely that $|E(G)| < k(n-1)$ implies $\bar{\tau}(G) < k$, is false in general. Let $k\ge 2$ and $n\ge 2k+1$ be integers, and let $G = K_{2k} \cup (n-2k)K_1$. Then $|E(G)|=k(2k-1)<k(n-1)$, whereas $\bar{\tau}(G)\geq \tau(K_{2k})=k$, providing a counterexample.
	 
	Let $k=1$ and $n\geq 2$ be integers, and let $G$ be a simple graph of order $n$. If $|E(G)|\geq n-1$, then as $G$ is simple, $G$ admits at least one edge, which forms a spanning tree of a subgraph of its $2$-vertex induced subgraph. Thus $\bar{\tau}(G)\geq 1$. Combining this with the contrapositive of Lemma~\ref{size-bound}, we obtain the following corollary.
    
	\begin{cor}\label{size-sufficient}
	Let $k\geq 1$ and $n\geq 2$ be integers, and let $G$ be a graph of order $n$. If $|E(G)|\geq  k(n-1)$, then $\bar{\tau}(G) \geq k$.
	\end{cor}
	
	The following lemma establishes bounds on the edge connectivity of $\tau_k$-maximal graphs.
	
	\begin{lem}\label{edge connectivity}
		Let $k\geq 1$ and $n\geq 2k+2$ be integers. If $G$ is a $\tau_k$-maximal graph of order $n$, then $k+1 \leq \kappa'(G) \leq 2k+1$. 
	\end{lem}
	
	\begin{proof}
		We prove that $\kappa'(G)\leq 2k+1$ by contradiction. Suppose that $\kappa'(G)\geq 2k+2=2(k+1)$. Then by Theorem~\ref{GT}, $\tau(G)\geq k+1$, contradicting the definition of a $\tau_k$-maximal graph. Hence $\kappa'(G)\leq 2k+1$.
		
		It remains to prove that $\kappa'(G)\geq k+1$. Since $G$ is $\tau_k$-maximal, we have $\bar{\tau}(G)\leq k$. If $G$ is a complete graph, then $\lfloor \frac{n}{2}\rfloor = \tau(G) \le \bar{\tau}(G) \leq k$, which implies $n\leq 2k+1$, a contradiction. Consequently, $G$ is not complete.
		
	    Suppose that $\kappa'(G)=t\leq k$. Then there exists an edge set $E_{t}=\{e_1,\ldots, e_t\}$ such that $G-E_t$ is disconnected. Let $G_1$ be a component of $G-E_t$ and $G_2=G-V(G_1)$. Since $|V(G_1)|\cdot|V(G_2)|\geq n-1\geq 2k+1 > k \geq t$, we have $E_{G^c}[V(G_1), V(G_2)]\neq \emptyset$. 
        
        Pick an arbitrary edge $e_0\in E_{G^c}[V(G_1),V(G_2)]\subseteq E(G^c)$. Since $G$ is $\tau_k$-maximal, there exists a subgraph $H\subseteq G+e_0$ such that $\tau(H)=\bar{\tau}(G+e_0)\geq k+1$ and $e_0 \in E(H)$. Therefore, $H$ contains $k+1$ edge-disjoint spanning trees, denoted by $T_0, T_1, \ldots, T_k$.
        
        Since every spanning tree $T_i$ contains at least one edge belonging to $E_t\cup \{e_0\}$, and these spanning trees are pairwise edge-disjoint, we have $t+1\geq k+1$, which forces $t\ge k$. Together with $t\le k$, we obtain $t=k$. Consequently, $H$ contains all edges of $E_k\cup\{e_0\}$, and each $T_i$ contains exactly one of them, where $i=1,2,\ldots, k$.
        
        Without loss of generality, assume $e_i \in E(T_i)$ for all $i=0, 1,\ldots, k$. Let $V_1=V(H)\cap V(G_1)$ and $V_2=V(H)\cap V(G_2)$. Then one of $V_1$ and $V_2$ contains at least two vertices. Assume, without loss of generality, $|V_1|\geq 2$. Thus, for $i=0, 1,\ldots, k$, each $T_i[V_1]$ is a spanning tree of $G[V_1]$. This forces $\tau(G[V_1])\geq k+1$. Since $G[V_1]$ is a subgraph of $G$, we have $\bar{\tau}(G)\geq \tau(G[V_1])\geq k+1$, which contradicts $\bar{\tau}(G)\leq k$. Therefore, $\kappa'(G)\ge k+1$.
    \end{proof}
    
    In particular, if $G$ is $\tau_1$-maximal, then $2\le \kappa'(G)\le 3$. The next two lemmas characterize the structure of $\tau_1$-maximal graphs with respect to their minimum edge cuts.
    
    \begin{lem}\label{edge conn-2}
    	Let $G$ be a $\tau_1$-maximal graph of order $n\geq 3$ with edge connectivity $\kappa'(G)=2$. Let $F$ be a minimum edge cut of $G$, and let $G_1$ and $G_2$ be the two components of $G-F$. Then exactly one of $G_1$ and $G_2$ is isomorphic to $K_1$, and the other is $\tau_1$-maximal.
    \end{lem}
	
	\begin{proof}
		Firstly, we prove that exactly one of $G_1$ and $G_2$ is isomorphic to $K_1$. Suppose that $|V(G_1)|\geq 2$ and $|V(G_2)|\geq 2$. Since $\kappa'(G)=2$, we have $|F|=2$. Combining this with $|V(G_1)|\cdot |V(G_2)|>2$, we obtain $E_{G^c}[V(G_1), V(G_2)]\neq \emptyset$.
		
		Pick an edge $e\in E_{G^c}[V(G_1), V(G_2)]\subseteq E(G^c)$. Since $G$ is $\tau_1$-maximal, there exists a subgraph $H\subseteq G+e$ such that $\tau(H)=\bar{\tau}(G+e)\geq 2$ and $e\in E(H)$. Let $H_1=H[V(H)\cap V(G_1)] \subseteq G$ and $H_2=H[V(H)\cap V(G_2)]\subseteq G$. Since $G$ is $\tau_1$-maximal, we have $\bar{\tau}(H_1)\leq \bar{\tau}(G)\leq 1<2$ and $\bar{\tau}(H_2)\leq \bar{\tau}(G)\leq 1<2$. 
		
		By Lemma~\ref{size-bound}, we get $|E(H_1)|\leq 2|V(H_1)|-3$ and $|E(H_2)|\leq 2|V(H_2)|-3$. Thus
			\begin{align*}
				|E(H)|& \leq |E(H_1)|+|E(H_2)|+|F\cup\{e\}|\\
				&\leq 2|V(H_1)|-3+2|V(H_2)|-3+3\\
				&=2|V(H)|-3.
			\end{align*}
		Since $\tau(H)\geq 2$, it follows that $|E(H)|\geq 2|V(H)|-2$, a contradiction. Hence, at least one of $G_1$ and $G_2$ is a single vertex. Since $n\geq 3$, we have exactly one of $G_1$ and $G_2$ is isomorphic to $K_1$.
		
		Without loss of generality, assume $V(G_1)=\{u\}$. It remains to show that $G_2$ is $\tau_1$-maximal. Since $G_2\subseteq G$ and $G$ is $\tau_1$-maximal, we have $\bar{\tau}(G_2)\leq \bar{\tau}(G) \leq 1$. If $G_2$ is a complete graph, then it is trivially $\tau_1$-maximal. If $G_2$ is not complete, then as $G$ is $\tau_1$-maximal, there exists a subgraph $H'\subseteq G+e$ such that $\tau(H')=\bar{\tau}(G+e)\geq 2$ for any edge $e\in E(G_2^c)$. 
		
		Let $T_1$ and $T_2$ be two edge-disjoint spanning trees of $H'$. If $u\notin V(H')$, then $H'\subseteq G_2+e$, and so $G_2$ is $\tau_1$-maximal. If $u\in V(H')$, then by $d_{H'}(u)\leq d_{G+e}(u)=2$, $u$ is a leaf in both $T_1$ and $T_2$. Therefore, $T_1-u$ and $T_2-u$ are also two edge-disjoint spanning trees of $H'-u\subseteq G_2+e$. Hence, $\bar{\tau}(G_2+e)\geq \tau(H'-u)=2$. Together with $\bar{\tau}(G_2)\le 1$, we obtain $G_2$ is $\tau_1$-maximal. 
	\end{proof}
	
	\begin{lem}\label{edge conn-3}
		Let $G$ be a $\tau_1$-maximal graph of order $n\geq 4$ with edge connectivity $\kappa'(G)=3$. Then there exists a vertex $u\in V(G)$ with $d_{G}(u)=3$. Furthermore, there is an edge $e_0 \in E(G_2 ^c)$ such that $G_2+e_0$ is $\tau_1$-maximal, where $G_2=G-u$.
	\end{lem}
	
	\begin{proof}
		We prove this lemma by establishing three claims sequentially.
		\begin{Cla}\label{cl21}
			There exists a vertex $u\in V(G)$ with $d_G(u)=3$. 
		\end{Cla}
		\noindent\emph{Proof of Claim~\ref{cl21}.} Suppose that $\delta(G)\geq 4$. Then $|E(G)|\geq \frac{\delta(G)|V(G)|}{2}=2n$. Since $G$ is $\tau_1$-maximal, it follows that $\bar{\tau}(G)\leq 1<2$. By Lemma~\ref{size-bound}, $|E(G)|<2n-2<2n$, a contradiction. Therefore, $\delta(G)\leq 3$. Together with $\delta(G) \geq \kappa'(G) =3$, we conclude that $\delta(G)=3$, so $G$ must contain a vertex of degree $3$. \hfill\qedsymbol 
		
		Let $u\in V(G)$ be a vertex with $d_{G}(u)=3$. Since $|E_G(u)|=3=\kappa'(G)$, $E_G(u)$ is a minimum edge cut of $G$, and $G-E_G(u)$ has exactly two components $G_1=G[\{u\}]$ and $G_2=G-u$.
		
		\begin{Cla}\label{cl22}
			There is an edge $e_0\in E(G_2^c)$ such that $\bar{\tau}(G_2+e_0)\leq 1$.
		\end{Cla}
		
		\noindent\emph{Proof of Claim~\ref{cl22}.} Let $N_G(u)=\{v_1,v_2,v_3\}$. Since $\bar{\tau}(G)\leq 1$ and $\tau(K_4)=2$, $G$ is $K_4$-free. Therefore, $G_2[\{v_1,v_2,v_3\}]$ is not a triangle. Without loss of generality, assume $v_1 v_2\notin E(G)$.
		
		Suppose that $\bar{\tau}(G_2+e)\geq 2$ for every $e\in E(G_2^c)$. Since $v_1 v_2\in  E(G_2^c)$, there exists a subgraph $H_1\subseteq G_2+v_1 v_2$ such that $\tau(H_1)\geq 2$ and $v_1v_2\in E(H_1)$. Because every tree of order $n\geq 2$ has exactly $n-1$ edges, we get $|E(H_1)|\geq 2|V(H_1)|-2$. Hence, we consider the following three cases.
		
		\noindent{\bf Case 1.} $v_3 \in V(H_1).$
		
		Let $H_1'=G[V(H_1)\cup \{u\}]\subseteq G$. Then 
		\begin{align*}
			|E(H_1')|&=|E(H_1)|-|\{v_1v_2\}|+|E_{H_1'}(u)|\\
			& \geq 2|V(H_1)|-2-1+3\\
			&= 2|V(H_1')|-2.
		\end{align*} 
		By Corollary~\ref{size-sufficient}, we get $\bar{\tau}(H_1')\geq 2$, and hence $\bar{\tau}(G) \geq \bar{\tau}(H_1')\geq 2$, which contradicts $\bar{\tau}(G)\leq 1$.

		\noindent{\bf Case 2.} $v_3\notin V(H_1)$ and $E_{G^c}[\{v_3\}, V(H_1)]=\emptyset$.
		
		Let $H_1''=G[V(H_1)\cup \{u,v_3\}]\setminus \{v_1 v_2\}\subseteq G$. Then 
		\begin{align*}
			|E(H_1'')|&=|E(H_1)|-|\{v_1v_2\}|+|E_{H_1''}(u) \cup E_{H_1''}(v_3)|\\
			&\geq 2|V(H_1)|-2-1+d_{H_1''}(u)+d_{H_1''}(v_3)-1\\
			&= 2|V(H_1)|-2-1+3+|V(H_1)|-1\\
			&> 2|V(H_1)|+2\\
			&=2|V(H_1'')|-2.
		\end{align*}	
		By Corollary~\ref{size-sufficient}, we get $\bar{\tau}(H_1'')\geq 2$, and hence $\bar{\tau}(G) \geq \bar{\tau}(H_1'')\geq 2$, a contradiction.
			\begin{figure}[!htb]
			\centering
			\begin{minipage}[b]{0.48\textwidth}
				\centering
				\includegraphics[
				width=0.95\textwidth,
				trim=4cm 17cm 4cm 4cm,  
				clip,
				keepaspectratio
				]{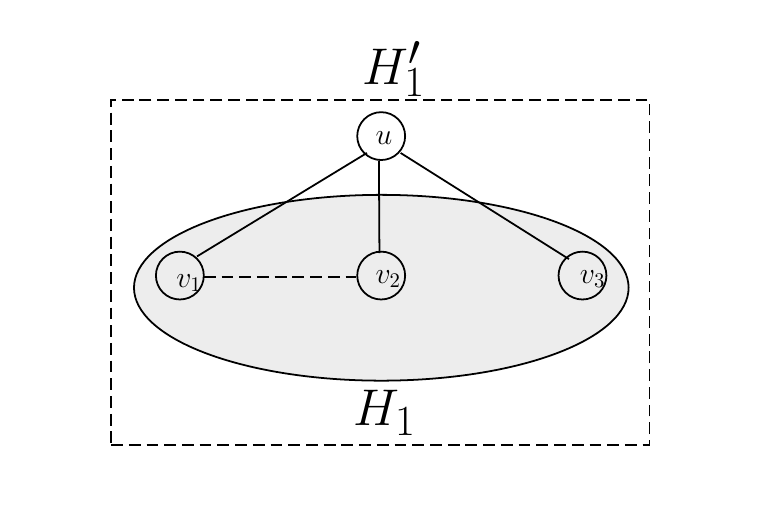}
				\caption{Structure for Case 1 of Lemma~\ref{edge conn-3}}
				\label{fig:case1}
			\end{minipage}
			\hfill
			\begin{minipage}[b]{0.48\textwidth}
				\centering
				\includegraphics[
				width=0.95\textwidth,
				trim=4cm 17cm 4cm 4cm,  
				clip,
				keepaspectratio
				]{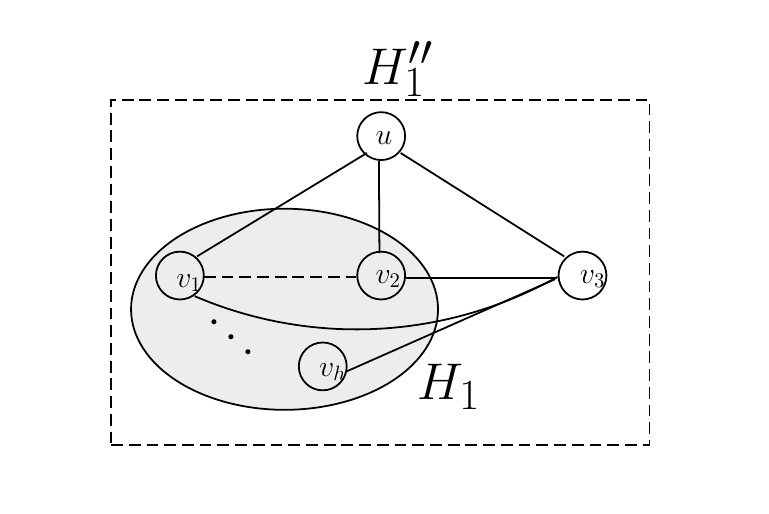}
				\caption{Structure for Case 2 of Lemma~\ref{edge conn-3}}
				\label{fig:case2}
			\end{minipage}		
		\end{figure}
				\begin{figure}[!htb]
			\centering
			\begin{minipage}[b]{0.48\textwidth}
				\centering
				\includegraphics[
				width=1.1\textwidth,
				trim=2.4cm 17.2cm 4cm 4cm,  
				clip,
				keepaspectratio
				]{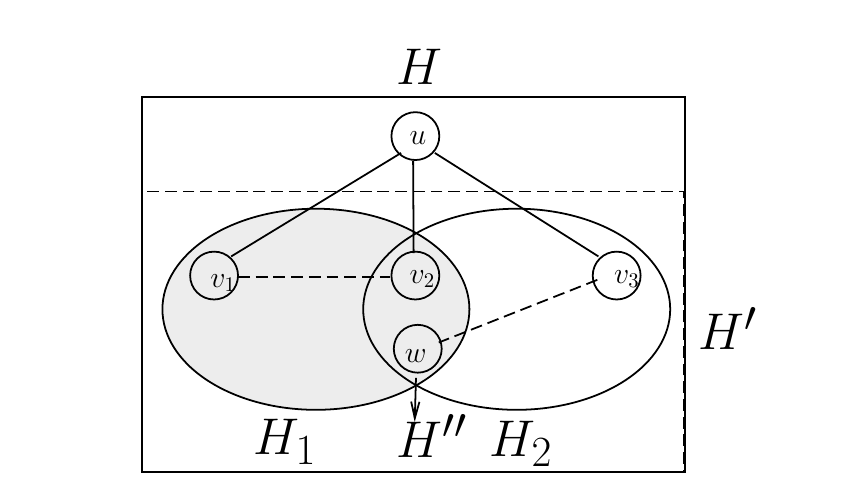}
				\caption{Structure for Case 3 of Lemma~\ref{edge conn-3}}
				\label{fig:case3}
			\end{minipage}
			\hfill
			\begin{minipage}[b]{0.48\textwidth}
				\centering
				\includegraphics[
				width=1.11\textwidth,
				trim=4cm 18cm 4cm 4cm,  
				clip,
				keepaspectratio
				]{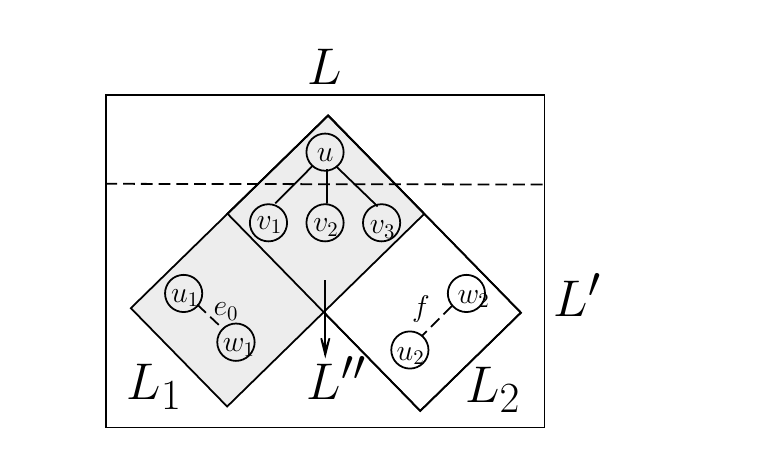}
				\caption{Structure for the proof of Claim~\ref{cl23}}
				\label{fig:case3b}
			\end{minipage}
		\end{figure}
		
	    \noindent{\bf Case 3.} $v_3\notin V(H_1)$ and $E_{G^c}[\{v_3\}, V(H_1)]\neq \emptyset$. 
		
	    Pick an edge $v_3 w\in E_{G^c}[\{v_3\}, V(H_1)]\subseteq E(G_2^c)$. By our assumption, $\bar{\tau}(G_2+v_3 w)\geq 2$, so there exists a subgraph $H_2\subseteq G_2+v_3 w$ such that $\tau(H_2)\geq 2$, and so $|E(H_2)|\geq 2|V(H_2)|-2$.
		
	    Let $H'=G[V(H_1)\cup V(H_2)]$, $H''=G[V(H_1)\cap V(H_2)]$ and $H=G[V(H')\cup \{u\}]$. Since $H''\subseteq G$, we have $\bar{\tau}(H'')\le 1<2$. By Lemma~\ref{size-bound}, we get $|E(H'')|\leq 2|V(H'')|-3$. Therefore, 
		\begin{align*}
			|E(H)|&=|E(H')|+|E_{H}(u)|\\
			&\geq |E(H_1)\setminus \{v_1v_2\}|+|E(H_2)\setminus\{v_3w\}|-|E(H'')|+3\\
			&\geq 2|V(H_1)|-3+2|V(H_2)|-3-2|V(H'')|+3+3\\
			&=2|V(H')|\\
			&=2|V(H)|-2.
		\end{align*}
	   By Corollary~\ref{size-sufficient}, we get $\bar{\tau}(H)\geq 2$, and hence $\bar{\tau}(G) \geq \bar{\tau}(H)\geq 2$, a contradiction.
          	
	   In all cases we obtain a contradiction. Therefore, there is an edge $e_0\in E(G_2^c)$ such that $\bar{\tau}(G_2+e_0)\leq 1$. \hfill\qedsymbol
	    
	   By Claim~\ref{cl22}, there is an edge $e_0\in E(G_2^c)$ satisfying $\bar{\tau}(G_2+e_0)\leq 1$. Since $G$ is $\tau_1$-maximal, there exists a subgraph $L_1 \subseteq G+e_0$ such that $\tau(L_1)=\bar{\tau}(G+e_0)\geq 2$, and so $|E(L_1)|\geq 2|V(L_1)|-2$. 
	
	   If $u\notin V(L_1)$ then $L_1\subseteq G_2+e_0$, contradicting $\bar{\tau}(G_2+e_0)\le 1$. Thus $u\in V(L_1)$. If $d_{L_1}(u)\le 2$, then $L_1-u\subseteq G_2+e_0$ and $\tau(L_1-u)\ge 2$, again a contradiction. Hence, $d_{L_1}(u)=3$.  
		
	    \begin{Cla}\label{cl23}
	    	For any $f\in E(G_2^c)\setminus \{e_0\}$, we have $\bar{\tau}(G_2+e_0+f)\geq 2$.
	    \end{Cla} 
	    
	    \noindent\emph{Proof of Claim~\ref{cl23}.} Since $G$ is $\tau_1$-maximal and $f\in E(G_2^c)$, there exists a subgraph $L_2\subseteq G+f$ such that $\tau(L_2)=\bar{\tau}(G+f)\geq 2$ and $f\in E(L_2)$. Thus $|E(L_2)|\geq 2|V(L_2)|-2$. 
	    If $u\notin L_2$, then $L_2\subseteq G_2+e_0+f$, so $\bar{\tau}(G_2+e_0+f)\geq \tau(L_2)\geq 2$, and the result holds. 
	    
	    If $u\in V(L_2)$, then by the same analysis as for $L_1$, we have $d_{L_2}(u)=3$. Let $L'=G[V(L_1)\cup V(L_2)\setminus\{u\}]\subseteq G_2$, $L''=G[V(L_1)\cap V(L_2)]$ and $L=L'\cup \{e_0\}\cup \{f\}\subseteq G_2+e_0+f$. Since $L''\subseteq G$, we have $|E(L'')|\leq 2|V(L'')|-3$. Therefore,
	    \begin{align*}
		|E(L)|&=|E(L')|+2\\
		&\geq |E(L_1)|-1+|E(L_2)|-1-|E(L'')|-|E_{G}(u)|+2\\
		&\geq 2|V(L_1)|-3+2|V(L_2)|-3-2|V(L'')|+3-3+2\\
		&=2|V(L)\cup \{u\}|-4\\
		&= 2|V(L)|-2.
    	\end{align*}
    	By Corollary~\ref{size-sufficient}, we obtain $\bar{\tau}(L)\geq 2$, and hence $\bar{\tau}(G_2+e_0+f)\geq  \bar{\tau}(L)\geq 2$.\hfill\qedsymbol
    	
    	Combining Claims~\ref{cl21},~\ref{cl22} and~\ref{cl23}, we conclude the result holds. 
	\end{proof}

	\section{Extremal size of $\tau_k$-maximal graphs}\label{Se3}
	
	We begin this section with an upper bound on the size of $\tau_k$-maximal graphs, which follows directly from Corollary \ref{size-sufficient}.	
	
	\begin{thm}\label{upper bound}
		Let $k\geq 1$ and $n\geq 2$ be integers. If $G$ is a $\tau_k$-maximal graph of order $n$, then 
		\[
		|E(G)|\leq (k+1)(n-1)-1.
		\]
	\end{thm}
	
	\begin{proof}
		Since $G$ is $\tau_k$-maximal, we have $\bar{\tau}(G)\leq k$. 
		Suppose that $|E(G)|\geq (k+1)(n-1)$. 
		Then by Corollary~\ref{size-sufficient}, $\bar{\tau}(G)\geq k+1$, a contradiction.
	\end{proof}
	
	Theorem~\ref{upper bound} shows that $(k+1)(n-1)-1$ is an upper bound on the size of $\tau_k$-maximal graph on $n$ vertices. We now construct a family of $\tau_k$-maximal graphs with prescribed edge connectivity that attain this bound. Our construction relies heavily on Harary graphs. 
	
	A \emph{Harary graph} $H_{k,n}$~\cite{Harary1962} is a $k$-connected graph on $n$ vertices with the minimum possible number of edges. It is well known that $H_{k,n}$ has exactly $\lceil\frac{kn}{2}\rceil$ edges. When $k\geq 2$ is even, $H_{k,n}$ is Hamiltonian, and thus its \emph{matching number} of $H_{k,n}$ is $\lfloor\frac{n}{2}\rfloor$.
    
    Let $G_1$ and $G_2$ be two vertex-disjoint graphs. The \emph{union} of $G_1$ and $G_2$, denoted by $G_1\cup G_2$, is the graph with vertex set $V(G_1)\cup V(G_2)$ and edge set $E(G_1)\cup E(G_2)$. The \emph{join} of $G_1$ and $G_2$, denoted by $G_1\vee G_2$, is the graph with vertex set $V(G_1)\cup V(G_2)$ and edge set \[E(G_1\vee G_2)=E(G_1)\cup E(G_2)\cup\{uv:u\in V(G_1), v\in V(G_2)\}.\]  
    
    \begin{Def}\label{construction}
      Let $n$, $k$ and $l$ be positive integers. Define the graph $G(k,l;n)$ as follows.
    \begin{enumerate}
     \item[\rm{(i)}] If $1\leq l \leq k$ and $n\geq 2k+2$, then 
    \[G(k,l;n)=K_{k-l+2}\vee \left(K_{k+l}^- \cup H_{2(l-1),n-2k-2}\right),\]
    where $K_{k+l}^-$ denotes the graph obtained from the complete graph $K_{k+l}$ by deleting exactly an edge. 
      
    \item[\rm{(ii)}] If $l=k+1$ and $n\geq 2k+4$, then 
    \[G(k,l;n)=H_{2(k+1),n} - M,\]
    where $M$ is a matching in $H_{2(k+1),n}$ with $|M|=k+2$.
    \end{enumerate}
       \end{Def}
       
 Let $\mathcal{G}(k,l)$ denote the family of all graphs $G(k,l;n)$ of order $n$ as defined above. We will show that every graph in this family is $\tau_k$-maximal and has edge connectivity exactly $k+l$. Now we recall some results on edge connectivity of graphs.
 
\begin{thm}[Plesn\'{\i}k~\cite{Plesnik1975}]\label{diameter condition}
  Let $G$ be a connected graph of diameter at most $2$. Then $\kappa'(G)=\delta(G)$.
\end{thm}

Since the join of any two graphs has diameter at most $2$, Theorem~\ref{diameter condition} immediately yields the following corollary.

\begin{cor}\label{edge connectivity of join}
  Let $G_1$ and $G_2$ be two vertex-disjoint graphs. Then \[\kappa'(G_1\vee G_2)=\delta(G_1\vee G_2)=\min\{\delta(G_1)+|V(G_2)|,\delta(G_2)+|V(G_1)|\}.\]
\end{cor}

 Let $G$ be a connected graph, and let $g$ be a positive integer. An edge set $F\subseteq E(G)$ is a \emph{$g$-restricted edge cut} of $G$ if $G-F$ is disconnected and every component of $G-F$ contains at least $g$ vertices. The \emph{$g$-restricted edge connectivity} of $G$, denoted by $\kappa'_g(G)$, is the cardinality of a minimum $g$-restricted edge cut. Define \[\xi_g(G)=\min\{|E_{G}(S)|:\emptyset\neq S\subseteq V(G),~|S|=g~\text{and}~G[S]~\text{is connected}.\}\]
 A graph $G$ is \emph{$\kappa'_g$-optimal} if $\kappa'_g(G)=\xi_g(G)$.

\begin{thm}[Liu, Huang and Zhang~\cite{Liu2011}] \label{restricted connectivity}
	Let $G=H_{k,n}$ be a Harary graph satisfying $k=2r$ is even. Then $G$ is $\kappa'_g$-optimal for any $1\leq g \leq \frac{n}{2}$. Furthermore	
\begin{align*}
		\kappa'_g(G)=
	\begin{cases}
		2rg - g(g-1), & \text{when } g<r,\\
		r(r +1), & \text{when } r \le g \le \dfrac{n}{2}.
	\end{cases}
\end{align*} 
\end{thm}

The following observation will also be used later.
 
 \begin{Obs}\label{Combinatorial Identities}
 	Let $a$ and $b$ be two positive integers. If $a$, $b\geq 1$, then $\binom{a+b}{2}=\binom{a}{2}+\binom{b}{2}+ab$.
 \end{Obs}
With the above preparations, we prove the following result.

\begin{lem}\label{main lemma}
Let $n$, $k$ and $l$ be positive integers satisfying $1 \leq l \leq k$ and $n\geq 2k+2$, or $l=k+1$ and $n\geq 2k+4$. Then for any $G\in \mathcal{G}(k,l)$, the following statements hold:
\begin{enumerate}
    \item[\rm{(i)}] $\kappa'(G)=k+l$;
    \item[\rm{(ii)}] $|E(G)|=(k+1)(n-1)-1$;
    \item[\rm{(iii)}] $G$ is $\tau_k$-maximal.
\end{enumerate}
\end{lem}

\begin{proof}
We proceed by considering the two cases $1\leq l\leq k$ and $l=k+1$ separately. 

\noindent{\bf Case 1.} $1\leq l \leq k$.

Let $H=H_{2(l-1), n-2k-2}$. Then by Definition~\ref{construction}, $G=K_{k-l+2}\vee \left(K_{k+l}^- \cup H\right)$. We verify the following three statements.

\noindent{\rm{(i)}} By the definition of $G$ and Corollary~\ref{edge connectivity of join}, we obtain \begin{align*}
      \kappa'(G)&=\min\{\delta(K_{k-l+2})+|V(K_{k+l}^- \cup H)|,\delta(K_{k+l}^- \cup H)+|V(K_{k-l+2})|\}\\
      &=\min\{(k-l+1)+(k+l)+(n-2k-2),2(l-1)+(k-l+2)\}\\
      &=\min\{n-1,k+l\}\\
      &=k+l.
\end{align*}
\noindent{\rm{(ii)}} By Definition~\ref{construction} and Observation~\ref{Combinatorial Identities}, we obtain
\begin{align*}
|E(G)|&=|E(K_{k-l+2})|+|E(K_{k+l}^-)|+|E(H)|+|V(K_{k-l+2})|\cdot(|V(K_{k+l}^-)|+|V(H)|)\\
&=\binom{k-l+2}{2}+\left(\binom{k+l}{2}-1\right)+(l-1)(n-2k-2)+(k-l+2)(k+l+n-2k-2)\\
&=\binom{2k+2}{2}+(n-2k-2)(k-l+2+l-1)-1\\
&=(k+1)(2k+1)+(k+1)(n-2k-2)-1\\
&=(k+1)(n-1)-1.
\end{align*}

\noindent{\rm{(iii)}} We show that $\bar{\tau}(G)\leq k$ by contradiction. Suppose that there exists a subgraph $L\subseteq G$ with $\tau(L)\geq k+1$, then $|E(L)|\geq (k+1)(|V(L)|-1)$.
		
Let $A= V(K_{k-l+2})\cap V(L)$, $B= V(K_{k+l}^-)\cap V(L)$ and $C= V(H)\cap V(L)$. Then $V(L)=A\cup B \cup C$, and 
\[|E(L)|=|E(L[A])|+|E(L[B])|+|E(L[C])|+|A|(|B|+|C|).\]
Clearly, $0\leq |A|\leq k-l+2$, $0\leq |B|\leq k+l$ and $0\leq |C| \leq n-2k-2$.

If $|B|<k+l$, then $|A|+|B|< 2k+2$, and so
\begin{align*}
	|E(L)|&\leq \binom{|A|}{2}+\binom{|B|}{2}+(l-1)|C|+|A|(|B|+|C|)\\
	&=\binom{|A|+|B|}{2}+(|A|+l-1)|C|\\
	&=\frac{(|A|+|B|)(|A|+|B|-1)}{2}+(|A|+l-1)|C|\\
	&<(k+1)(|A|+|B|-1)+(|A|+l-1)|C|\\
	&\leq (k+1)(|A|+|B|-1)+(k+1)|C|\\
	&=(k+1)(|A|+|B|+|C|-1)\\
	&=(k+1)(|V(L)|-1),
\end{align*}
a contradiction.

If $|B|=k+l$, then $|A|+|B|\leq 2k+2$, and so
\begin{align*}
|E(L)|&\leq \binom{|A|}{2}+\left(\binom{|B|}{2}-1\right)+(l-1)|C|+|A|(|B|+|C|)\\
&=\binom{|A|+|B|}{2}+(|A|+l-1)|C|-1\\
&\leq(k+1)(2k+1)+(|A|+l-1)|C|-1\\
&\leq(k+1)(2k+1)+(k+1)|C|-1\\
&=(k+1)(|V(L)|-1)-1,
\end{align*}
a contradiction.  Hence no such $L$ exists, and $\bar{\tau}(G)\leq k$.

Since $|E(G)|=(k+1)(n-1)-1$, adding any edge $e$ to $G$, we get $|E(G+e)|=(k+1)(n-1)$. By Corollary~\ref{size-sufficient}, $\bar{\tau}(G+e)\geq k+1$. Combining this with $\bar{\tau}(G)\leq k$, we get $G$ is $\tau_k$-maximal.

\noindent{\bf Case 2.} $l= k+1$.

Let $H' = H_{2(k+1),n}$ be a Harary graph. Then $H'$ is $2(k+1)$-regular and $G = H' - M$.

\noindent{\rm{(i)}} Since removing a matching $M$ with $|M|=k+2$ reduces the degree of exactly $2(k+2)$ vertices by $1$, we have \[\delta(G) = 2(k+1) - 1 = 2k+1 = k+l.\]
Therefore, $\kappa'(G) \leq \delta(G)\leq k+l$.

It remains to show that $\kappa'(G) \geq k+l=2k+1$. Assume there exists a proper nonempty subset $X \subset V(G)$ such that the edge cut $E_G(X)$ satisfies $|E_G(X)| \leq 2k$.
Let $|E_{H'}(X) \cap M|=t$. Then $t\leq |M|=k+2$ and
\[
|E_{H'}(X)| = |E_G(X)| + t \leq 2k + t.
\]
Since $H'$ is $2(k+1)$-edge-connected, $|E_{H'}(X)| \geq 2k+2$, which implies $t \geq 2$. We consider the following two subcases.

\noindent{\bf Subcase 2.1.} $|X| = 1$ or $|V(G)\setminus X| = 1$.

Without loss of generality, assume $|X| = 1$. Since $M$ is a matching, each vertex in $X$ is incident to at most one edge in $M$, it follows that $t \leq 1$, contradicting $t \geq 2$.

\noindent{\bf Subcase 2.2.} $|X| \geq 2$ and $|V(G)\setminus X| \geq 2$.

Applying Theorem~\ref{restricted connectivity} to $H'$ with $r=k+1$ and $g\geq 2$, we obtain $E_{H'}(X)\geq 4k+2$. Therefore,
\[
|E_G(X)| = |E_{H'}(X)| - t \geq (4k+2) - (k+2) = 3k >2k,
\]
contradicting $|E_G(X)| \leq 2k$.

Hence $\kappa'(G) \geq 2k+1=k+l$.
Combining with $\kappa'(G) \leq \delta(G) = k+l$, we conclude that $\kappa'(G) = k+l$.

\noindent{\rm{(ii)}} By a simple calculation, we get
\begin{align*}
|E(G)|&=|E(H')|-|M|\\
&=(k+1)n-(k+2)\\
&=(k+1)(n-1)-1.
\end{align*}

\noindent{\rm{(iii)}} We prove that $\bar{\tau}(G)\leq k$ by contradiction. Assume there exists a subgraph $L'\subseteq G$ with $\tau(L')\geq k+1$. Thus
\[|E(L')|\geq (k+1)(|V(L')|-1).\]

By the fact that \[k+1 \leq \tau(L') \leq \tau(K_{|V(L')|}) =\lfloor\tfrac{|V(L')|}{2}\rfloor,\] we get $|V(L')|\geq 2k+2$.

If $2k+2\leq |V(L')|\leq n-2$, then by the handshaking lemma and Theorem~\ref{restricted connectivity}, we have 
\begin{align*}
	|E(L')|&=\frac{1}{2}\sum\nolimits_{v\in V(L')}d_{L'}(v)\\
	&\leq \frac{1}{2}\Big[\sum\nolimits_{v\in V(L')}d_{H'}(v)-\big|E_{H'}(V(L'))\big|\Big]\\
	&\leq \frac{1}{2}\Big[(2k+2)|V(L')|-(4k+2)\Big]\\
	&= (k+1)|V(L')|-(2k+1)\\
	&< (k+1)\big(|V(L')|-1\big).
\end{align*}
a contradiction.

If $|V(L')|=n-1$, let $v$ be the unique vertex in $V(G)\setminus V(L')$. Then $|E_{G}(v)|=d_{G}(v)\geq \delta(G)=2k+1$, and thus
\begin{align*}
	|E(L')|&=|E(G)|-|E_{G}(v)|\\
	&\leq (k+1)(n-1)-1-(2k+1)\\
	&<(k+1)(|V(L')|-1),
\end{align*}
which again yields a contradiction.

If $|V(L')|=n$, then $L'=G$, and hence
\[|E(L')|=|E(G)|=(k+1)(n-1)-1<(k+1)(|V(L')|-1),\]
a contradiction. Therefore, $\bar{\tau}(G)\leq k$.

 Since $|E(G)|=(k+1)(n-1)-1$, adding any edge $e\in E(G^c)$ to $G$ yields $|E(G+e)|=(k+1)(n-1)$. By Corollary~\ref{size-sufficient}, we have $\bar{\tau}(G+e)\geq k+1$. Combining with $\bar{\tau}(G)\leq k$,
 we have $G$ is $\tau_k$-maximal.
\end{proof}

The above construction shows that for any integers $k\ge 1$ and $1\le l\le k+1$, there exists a $\tau_k$-maximal graph having edge connectivity exactly $k+l$. Together with Lemma~\ref{edge connectivity}, this implies that the bounds $k+1\le \kappa'(G)\le 2k+1$ are sharp for all $k$. 

Theorem~\ref{upper bound} and the existence of extremal graphs in $\mathcal{G}(k,l)$ motivate the following conjecture.

\begin{conj}\label{extremal size}
	Let $k\geq 1$ and $n\geq 2k+2$ be integers. If $G$ is a $\tau_k$-maximal graph of order $n$, then $|E(G)|=(k+1)(n-1)-1$.
\end{conj}

The following theorem confirms Conjecture~\ref{extremal size} for the case $k=1$.	
   \begin{thm}
	Let $G$ be a $\tau_1$-maximal graph of order $n \geq 2$. Then $|E(G)|=2n-3$.
    \end{thm}
	\begin{proof}
		By Theorem~\ref{upper bound}, we have $|E(G)|\leq 2n-3$.
		
		Now we prove that $|E(G)|\geq 2n-3$. By induction on $n$. If $n=2$, then $G \cong K_2$ and $|E(G)|=1=2\times 2-3$.
		
		Now assume that $n\geq 3$ and the result holds for all $\tau_1$-maximal graphs of order less than $n$. By Lemma~\ref{edge connectivity}, we have $2\le\kappa'(G)\le 3$. We consider the following two cases based on the edge connectivity of $G$.
		
		\noindent{\bf Case 1.} $\kappa'(G)=2$.
		
		By Lemma~\ref{edge conn-2}, there exists an edge cut $F$ with $|F|=2$ such that $G-F$ contains exactly two components $G_1$ and $G_2$. Without loss of generality, assume that $|V(G_1)|\leq |V(G_2)|$. Then by Lemma~\ref{edge conn-2}, $G_1$ is isomorphic to $K_1$ and $G_2$ is $\tau_1$-maximal. By the induction hypothesis, $|E(G_2)|\geq 2|V(G_2)|-3=2(n-1)-3$. Therefore,	
	\begin{align*}
		|E(G)|&=|E(G_1)|+|E(G_2)|+|F|\\
		& \geq 2(n-1)-3+2\\
		& =2n-3.
	\end{align*}
	
	\noindent{\bf Case 2.} $\kappa'(G)=3$.
	
	By Lemma~\ref{edge conn-3}, there exists a vertex $u$ with $d_G(u)=3$. Then $F'=E_G(u)$ is a minimum edge cut of $G$, and $G-F'$ has exactly two components $G_1'=G[\{u\}]$ and $G_2'=G-u$. By Lemma~\ref{edge conn-3}, there exists an edge $e_0\in E((G_2')^c)$ such that $G_2'+e_0$ is $\tau_1$-maximal. Applying the induction hypothesis to $G_2'+e_0$ yields
	$|E(G_2'+e_0)| \geq 2(n-1)-3$. Consequently, $|E(G_2')| \geq 2(n-1)-4$. Thus
	\begin{align*}
		|E(G)|&=|E(G_1')|+|E(G_2')|+|F'|\\
		& \geq 2(n-1)-4+3\\
		& =2n-3.
	\end{align*}
	 In both cases, we have $|E(G)|\ge 2n-3$.  Together with the upper bound $|E(G)|\le 2n-3$, we conclude
	$|E(G)| = 2n-3$ for all $n\geq 2$.
	\end{proof}

\end{document}